\documentclass[]{amsart}
\usepackage{amsmath}
\usepackage{amsthm}
\usepackage{amssymb}
\usepackage{amscd}
\usepackage{amsfonts}
\usepackage{amsbsy}
\usepackage{epsfig,afterpage}

\newtheorem {theorem} {Theorem}%[section]
\newtheorem {proposition} [theorem]{Proposition}

\newtheorem {remark} [theorem]{Remark}

\newtheorem {definition} [theorem]{Definition}

\usepackage[usenames,dvipsnames]{color}

\begin{document}

\title[Canards from Chua's circuit]
{Canards from Chua's circuit}

\author[J.M. Ginoux, J. Llibre]
{Jean-Marc Ginoux$^1$, Jaume Llibre$^2$ and Leon Chua$^3$}

\address{$^1$ Laboratoire {\sc Protee}, I.U.T. de Toulon,
Universit\'{e} du Sud, BP 20132, F-83957 La Garde cedex, France}
\email{ginoux@univ-tln.fr}

\address{$^2$ Departament de Matem\`{a}tiques,
Universitat Aut\`{o}noma de Barcelona, 08193 Bellaterra, Barcelona,
Spain} \email{jllibre@mat.uab.cat}

\address{$^3$ DEECS Department, University of California, Berkeley
253 Cory Hall {\#}1770, Berkeley, CA 94720-1770} \email{chua@eecs.berkeley.edu}



\begin{abstract}
At first, the aim of this work is to extend Beno\^{i}t's theorem for the generic existence of ``canards'' solutions in \textit{singularly perturbed dynamical systems} of dimension three with one fast variable to those of dimension four. Then, it is established that this result can be found according to the \textit{Flow Curvature Method}. Applications to Chua's cubic model of dimension three and four enables to state existence of ``canards'' solutions in such systems.
\end{abstract}

\maketitle

\section{Introduction}
\label{Intro}

Many systems in biology, neurophysiology, chemistry, meteorology, electronics exhibit several time scales in their evolution. Such systems, todays called \textit{singularly perturbed dynamical systems}, have been modeled by a system of differential equations (\ref{eq1}) having a small parameter multiplying one or several components of its vector field. Since the works of Andronov \& Chaikin [1937] and Tikhonov [1948], the \textit{singular perturbation method\footnote{For an introduction to singular perturbation method see Malley [1974] and Kaper [1999].}} has been the subject of many research, among which we will quote those of Arg\'{e}mi [1978] who carefully studied the slow motion. According to Tikhonov [1948], Takens [1976], Jones [1994] and Kaper [1999] \textit{singularly perturbed systems} may be defined as:

\begin{equation}
\label{eq1}
\begin{array}{*{20}c}
 {{\vec {x}}' = \varepsilon \vec {f}\left( {\vec {x},\vec {y},\varepsilon }
\right),\mbox{ }} \hfill \\
 {{\vec {y}}' =  \vec {g}\left( {\vec {x},\vec {y},\varepsilon }
\right)}, \hfill \\
\end{array}
\end{equation}

where $\vec {x} \in \mathbb{R}^p$, $\vec {y} \in \mathbb{R}^m$,
$\varepsilon \in \mathbb{R}^ + $, and the prime denotes
differentiation with respect to the independent variable $t$. The
functions $\vec {f}$ and $\vec {g}$ are assumed to be $C^\infty$
functions\footnote{In certain applications these functions will be
supposed to be $C^r$, $r \geqslant 1$.} of $\vec {x}$, $\vec {y}$
and $\varepsilon$ in $U\times I$, where $U$ is an open subset of
$\mathbb{R}^p\times \mathbb{R}^m$ and $I$ is an open interval
containing $\varepsilon = 0$.

\smallskip

In the case when $0 < \varepsilon \ll 1$, i.e., is a small positive number, the variable $\vec {x}$ is called \textit{slow} variable, and $\vec {y}$ is called \textit{fast} variable. Using Landau's notation: $O\left( \varepsilon^k \right)$ represents a function $f$ of $u$ and $\varepsilon $ such that $f(u,\varepsilon) / \varepsilon^k$ is bounded for positive $\varepsilon$  going to zero, uniformly for $u$ in the given domain.

\smallskip

In general it is used to consider that $\vec {x}$ evolves at an $O\left( \varepsilon \right)$ rate; while $\vec {y}$ evolves at an $O\left( 1 \right)$ \textit{slow} rate. Reformulating system (\ref{eq1}) in terms of the rescaled variable $\tau = \varepsilon t$, we obtain

\begin{equation}
\label{eq2}
\begin{aligned}
\dot {\vec {x}} & = \vec{f} \left( {\vec{x},\vec{y},\varepsilon} \right), \\
\varepsilon \dot {\vec {y}} & = \vec {g}\left( {\vec{x}, \vec{y},\varepsilon }
\right).
\end{aligned}
\end{equation}

The dot represents the derivative with respect to the new independent variable $\tau$.

\smallskip

The independent variables $t$ and $\tau $ are referred to the \textit{fast} and \textit{slow} times, respectively, and (\ref{eq1}) and (\ref{eq2}) are called the \textit{fast} and \textit{slow} systems, respectively. These systems are equivalent whenever $\varepsilon \ne 0$, and they are labeled \textit{singular perturbation problems} when $0 < \varepsilon \ll 1$. The label ``singular'' stems in part from the discontinuous limiting behavior
in system (\ref{eq1}) as $\varepsilon \to 0$.

\smallskip

In such case system (\ref{eq2}) leads to a differential-algebraic system called \textit{reduced slow system} whose dimension decreases from $p + m = n$ to $m$. Then, the \textit{slow} variable $\vec {x} \in \mathbb{R}^p$ partially evolves in the submanifold $M_0$ called the \textit{critical manifold}\footnote{It corresponds to the approximation of the slow invariant manifold, with an error of $O(\varepsilon)$.} and defined by

\begin{equation}
\label{eq3} M_0 := \left\{ {\left( {\vec {x},\vec {y}} \right):\vec
{g}\left( {\vec {x},\vec {y},0} \right) = {\vec {0}}} \right\}.
\end{equation}

When $D_{\vec{x}}\vec{f}$ is invertible, thanks to the Implicit Function Theorem, $M_0 $ is given by the graph of a $C^\infty $ function $\vec {x} = \vec {G}_0 \left( \vec {y} \right)$ for $\vec {y} \in D$, where $D\subseteq \mathbb{R}^p$ is a compact, simply connected domain and the boundary of D is an $(p - 1)$--dimensional $C^\infty$ submanifold\footnote{The set D is overflowing invariant with respect to (\ref{eq2}) when $\varepsilon = 0$. See Kaper [1999] and Jones [1994].}.

\smallskip

According to Fenichel [1971-1979] theory if $0 < \varepsilon \ll 1$ is sufficiently small, then there exists a function $\vec {G}\left( {\vec {y},\varepsilon } \right)$ defined on D such that the manifold

\begin{equation}
\label{eq4} M_\varepsilon := \left\{ {\left( {\vec {x},\vec {y}}
\right):\vec {x} = \vec {G}\left( {\vec {y},\varepsilon } \right)} \right\},
\end{equation}

is locally invariant under the flow of system (\ref{eq1}). Moreover, there exist perturbed local stable (or attracting) $M_a$ and unstable (or repelling) $M_r$ branches of the \textit{slow invariant manifold} $M_\varepsilon$. Thus, normal hyperbolicity of $M_\varepsilon$ is lost via a saddle-node bifurcation of the \textit{reduced slow system} (\ref{eq2}). Then, it gives rise to solutions of ``canard'' type that have been discovered by a group of French mathematicians (Beno\^{i}t \textit{et al.} [1981]) in the beginning of the eighties while they were studying relaxation oscillations in the classical equation of Van der Pol [1926] (with a constant forcing term). They observed, within a small range of the control parameter, a fast transition for the amplitude of the limit cycle varying suddenly from small amplitude to a large amplitude. Due to the fact that the shape of the limit cycle in the phase plane looks as a duck they called it ``canard cycle''. So a ``canard'' is a solution of a singularly perturbed dynamical system following the \textit{attracting} branch $M_a$ of the \textit{slow invariant manifold}, passing near a bifurcation point located on the fold of the \textit{critical manifold}, and then following the \textit{repelling} branch $M_r$ of the \textit{slow invariant manifold}.

\newpage

\begin{remark}
Geometrically a \textit{maximal canard} corresponds to the intersection of the attracting and repelling branches $M_a \cap M_r$ of the slow manifold in the vicinity of a non-hyperbolic point. Canards are a special class of solution of singularly perturbed dynamical systems for which normal hyperbolicity is lost.\\
Canards in singularly perturbed systems with two or more slow variables {\rm (}$\vec {x} \in \mathbb{R}^p$, $p \geqslant 2$ {\rm )} and one fast variable {\rm (}$\vec {y} \in \mathbb{R}^m$, $m = 1${\rm )} are robust, since maximal canards generically persist under small parameter changes\footnote{See Beno\^{i}t [1983, 2001], Szmolyan \& Wechselberger [2001] and Wechselberger [2005].}.

\end{remark}

In dimension three, Beno\^{i}t [1983] has stated a theorem for the existence of canards (recalled in Sec. 3) in which he proved that if the ``reduced vector field'' has a pseudo-singular point of saddle type (whose definitions are recalled in Sec. 2), then the ``full system'' exhibits a canard solution which evolves from the attractive part of the slow manifold towards its repelling part. So, the first aim of this work (presented in Sec. 4) is to extend this theorem to dimension four.

Then, it is also stated that such condition for the generic existence of the peculiar solutions, called ``canards'', in \textit{singularly perturbed dynamical systems} of dimension three and four with only one fast variable can be found according to the \textit{Flow Curvature Method} developed by Ginoux \textit{et al.} [2008] and Ginoux [2009] and recalled in Sec. 5.

Thus, we will establish that Beno\^{i}t's condition for the generic existence of ``canards'' solutions in such systems is also given by the existence of a pseudo-singular point of saddle type for the \textit{flow curvature manifold} of the ``reduced systems''. This result, presented in Sec. 5, is based on the use of the so-called ``Second derivative test'' involving the Hessian of hypersurfaces. Applications to Chua's cubic model of dimension three and four enables to state existence of ``canards'' solutions in such systems.

\section{Definitions}
\label{Sec2}

Let's consider a $n$-dimensional \textit{singularly perturbed dynamical system} which may be written as:

\begin{equation}
\label{eq5}
\begin{aligned}
\dot {\vec {x}} & = \vec{f} \left( {\vec{x},\vec{y},\varepsilon} \right), \\
\varepsilon \dot {\vec {y}} & = \vec {g}\left( {\vec{x}, \vec{y},\varepsilon }
\right),
\end{aligned}
\end{equation}

\smallskip

where $\vec{x}= (x_1, \ldots, x_p)^t \in \mathbb{R}^p$, $\vec{y}= (y_1, \ldots, y_m)^t \in \mathbb{R}^m$, $\vec{f}= (f_1, \ldots, f_p)^t$, $\vec{g}= (g_1, \ldots, g_m)^t$, $\varepsilon \in \mathbb{R}^ + $ such that $0 < \varepsilon << 1$, and the dot denotes differentiation with respect to the independent variable $t$. The functions $f_i$ and $g_i$ are assumed to be $C^2$ functions of $x_i$ and $y_j$ (with $1 < i < p$ and $1 < j < m$).

In order to tackle this problem many analytical approaches such as \textit{asymptotic expansions} and \textit{matching methods} were developed (see Zvonkin \& Schubin [1984] and Rossetto [1986]). According to O'Malley [1991] the asymptotic expansion is expected to diverge. Then, Beno\^{i}t [1982, 1983] used non-standard analysis to study canards in $\mathbb{R}^3$.\\

In the middle of the seventies, a geometric approach developed by Takens [1976] consisted in considering that the following system:

\begin{equation}
\label{eq6}
\begin{aligned}
\dot {\vec {x}} & = \vec{f} \left( {\vec{x},\vec{y},\varepsilon} \right), \\
0 & = \vec {g}\left( {\vec{x}, \vec{y},\varepsilon }
\right),
\end{aligned}
\end{equation}

which has been called \textit{constrained system} corresponds to the \textit{singular approximation} of system (\ref{eq5}) and where $\vec {g}\left( \vec{x}, \vec{y},\varepsilon \right) = 0$ defines the so-called \textit{slow manifold} $S_0$ or \textit{critical manifold} of the \textit{singular approximation}, \textit{i.e.} the zero order approximation in $\varepsilon$ of the \textit{slow manifold}.

\section{Three-dimensional singularly perturbed systems}
\label{Sec3}

In dimension greater than two, it is important to distinguish cases depending on \textit{fast} dimensions $m$ and \textit{slow} dimensions $p$. For three-dimensional \textit{singularly perturbed dynamical systems} we have two cases: $(p,m) = (2,1)$ and $(p,m) = (1,2)$. In this work we will only focus on the former case which has been subject of extensive research led by Eric Beno\^{i}t [1981, 1982, 1983, 2001] and summed up below. So, in the case $(p,m) = (2,1)$ three-dimensional \textit{singularly perturbed dynamical systems} (\ref{eq5}) may be defined as:

\begin{equation}
\label{eq7}
\begin{aligned}
 \dot{x_1} & = f_1 \left( x_1, x_2, y_1 \right), \hfill \\
 \dot{x_2} & = f_2 \left( x_1, x_2, y_1 \right), \hfill \\
 \varepsilon \dot{y_1} & = g_1 \left( x_1, x_2, y_1 \right), \hfill
\end{aligned}
\end{equation}

\smallskip

where $\vec{x}= (x_1, x_2)^t \in \mathbb{R}^2$, $\vec{y}= (y_1) \in \mathbb{R}^1$, $0 < \varepsilon << 1$ and the functions $f_i$ and $g_1$ are assumed to be $C^2$ functions of $(x_1, x_2, y_1)$.

\subsection{Fold, cusp and pseudo-singular points}\hfill\\

Let's recall the following definitions

\begin{definition}\label{def1} \hfill \\

\smallskip

The location of the points where $\partial_{y_1} g_1\left(x_1,x_2,y_1 \right) = p\left(x_1,x_2,y_1 \right) = 0$ and $g_1\left(x_1,x_2,y_1 \right) = 0$ is called the \textit{fold}.
\end{definition}

Following to Arg\'{e}mi [1978], the \textit{cofold} is defined as the projection, if it exists, of the fold line onto $S_0$   along the
$y_1$-direction.

\smallskip

According  to Beno\^{i}t [1983] system (\ref{eq7}) may have various types of singularities.

\begin{definition} \label{def2} \hfill \\

\begin{itemize}

\item[-] The fold is the set of points where the \textit{slow manifold}
is tangent to the $y_1$-direction.

\item[-] The cusp is the set of points where the \textit{fold} is
tangent to the $y_1$-direction.

\item[-] The stationary points are not on the \textit{fold} according to genericity assumptions.

\item[-] The pseudo-singular points are defined as the location of the points where

\begin{equation}
\label{eq8}
\begin{aligned}
& g_1\left(x_1,x_2,y_1 \right) = 0,\\
& \frac{\partial g_1\left(x_1,x_2,y_1 \right)}{\partial y_1} = 0, \\
& \frac{\partial g_1\left(x_1,x_2,y_1 \right)}{\partial x_1} f_1\left(x_1, x_2, y_1 \right) + \frac{\partial g_1\left(x_1,x_2,y_1 \right)}{\partial x_2} f_2\left(x_1,x_2,y_1 \right) = 0.
\end{aligned}
\end{equation}

\end{itemize}

\end{definition}

The concept of \textit{pseudo-singular points} has been originally introduced by Takens [1976] and Arg\'{e}mi [1978]. Again,
according to Beno\^{i}t [1983]:\\

\begin{itemize}

\item[-] the first condition indicates that the point belongs to the \textit{slow manifold},

\item[-] the second condition means that the point is on the \textit{fold},

\item[-] the third condition shows that the projection of the vector field on the
$(x_1,x_2)$-plane is tangent to the \textit{fold}.

\end{itemize}

\subsection{Reduced vector field}\hfill \\

If $x_1$ can be expressed as an implicit function of $x_2$ and $y_1$ defined by $g_1\left(x_1,x_2,y_1 \right) = 0$, the ``reduced normalized vector field'' reads:

\begin{equation}
\label{eq9}
\begin{aligned}
 \dot{x_2} & = - f_2 \left( x_1, x_2, y_1 \right) \frac{\partial g_1}{\partial y_1}\left( x_1, x_2, y_1 \right), \hfill \\
 \dot{y_1} & = \frac{\partial g_1}{\partial x_1} f_1\left(x_1, x_2, y_1 \right) + \frac{\partial g_1}{\partial x_2} f_2\left(x_1, x_2, y_1 \right).
\end{aligned}
\end{equation}

\subsection{Reduced vector field method}\hfill \\

By using the classification of fixed points of two-dimensional dynamical systems based on the sign of the eigenvalues of the functional Jacobian matrix, Beno\^{i}t [1983] characterized the nature of the \textit{pseudo-singular point} $M$ of the ``reduced vector field'' (\ref{eq9}).
Let's note $\Delta$ and $T$ respectively the determinant and the trace of the functional Jacobian matrix associated with system (\ref{eq9}). The \textit{pseudo-singular point} $M$ is:\\

\begin{itemize}

\item  a \textit{saddle} if and only if $\Delta < \dfrac{T^2}{4}$ and $\Delta < 0$.

\item  a \textit{node} if and only if $0 < \Delta < \dfrac{T^2}{4}$.

\item  a \textit{focus} if and only if $\dfrac{T^2}{4} < \Delta $.

\end{itemize}

Then, Beno\^{i}t [1983, p. 171] states the following theorem for the existence of canards:

\begin{theorem}\hfill \\
\label{theo4}
If the ``reduced vector field'' {\rm (\ref{eq9})} has a pseudo-singular point of saddle type, then system {\rm (\ref{eq7})} exhibits a canard solution which evolves from the attractive part of the slow manifold towards its repelling part.

\end{theorem}

\begin{proof}
See Beno\^{i}t [1983, p. 171].
\end{proof}

\subsection{Chua's system}\hfill \\

Let's consider the system introduced by Itoh \& Chua [1992]:

\begin{equation}
\label{eq10}
\begin{aligned}
\dot{x} & = z - y, \\
\dot{y} & = \alpha(x + y), \\
\varepsilon \dot{z} & = -x - k(z),
\end{aligned}
\end{equation}

where $k(z) = z^3/3 - z$ and $\alpha$ is a constant.\\

According to Eq. (\ref{eq9}) the reduced vector field reads:

\begin{equation}
\label{eq11}
\begin{aligned}
\dot{y} & = \alpha k'(z)(-k(z) + y) = \alpha (z^2 - 1)\left(-
\dfrac{z^3}{3} + z + y\right), \\
\dot{z} & = y - z.
\end{aligned}
\end{equation}

By Def. \ref{def2} the singularly perturbed dynamical system (\ref{eq10}) admits $M(\pm 2/3, \pm 1, \pm 1)$ as
\textit{pseudo-singular points}. The functional Jacobian matrix of reduced vector field (\ref{eq11}) evaluated at $M$ reads:

\begin{equation}
\label{eq12}
\begin{pmatrix}
0 & \dfrac{10\alpha}{3} \\
1 & -1
\end{pmatrix}
\end{equation}

from which we deduce that: $\Delta = -\dfrac{10\alpha}{3}$ and $T = -1$. So, we have:

\begin{equation}
\label{eq13}
\dfrac{T^2}{4} - \Delta = \dfrac{1}{12}\left( 3 + 40 \alpha \right).
\end{equation}

\smallskip

Thus, according to Th. \ref{theo4}, if $3 + 40 \alpha > 0$ and $\alpha > 0$, then $M$ is a \textit{pseudo-singular saddle point} and
so system (\ref{eq10}) exhibits canards solution. Itoh \& Chua [1992, p. 2791] have also noticed that if $\alpha > 0$, system (\ref{eq10}) has a \textit{pseudo-singular saddle point}.\\

Nevertheless, the original system (\ref{eq10}) admits, except the origin, two \textit{fixed points} $I(\pm \sqrt{6}, \mp \sqrt{6}, \mp \sqrt{6})$. The functional Jacobian matrix of the ``normalized slow dynamics'' evaluated at $I$ reads:

\begin{equation}
\label{eq14}
\begin{pmatrix}
0 & -5 & 5 \\
5\alpha & 5\alpha & 0\\
0 & 1 & -1
\end{pmatrix}
\end{equation}

from which we deduce that there are three eigenvalues:

\[
\lambda_1 = 0 \mbox{ ; } \lambda_{2,3} = \dfrac{1}{2} \left(-1 + 5 \alpha \pm \sqrt{1 - 90 \alpha + 25 \alpha^2} \right)
\]

\smallskip

Then, if these eigenvalues are complex conjugated we have:

\[
2Re\left( \lambda_{2,3}\right) =  -1 + 5\alpha
\]

\smallskip

But, according to the theorem of Lyapounov [1892], \textit{fixed points} $I$ are non stable equilibria provided that $-1 + 5\alpha > 0$.\\

Thus, ``canards'' solutions are observed in Chua's system (\ref{eq10}) for $\alpha > 1/5$ as exemplified in Fig. 1 in which such solutions passing through the \textit{pseudo-singular saddle point} $M(2/3, 1, 1)$ have been plotted for parameter set $(\alpha = 0.2571389636, \varepsilon = 1/20)$ in the $(x,y,z)$ phase space.

\begin{figure}[htbp]
\centerline{\includegraphics[width=8cm,height=8cm]{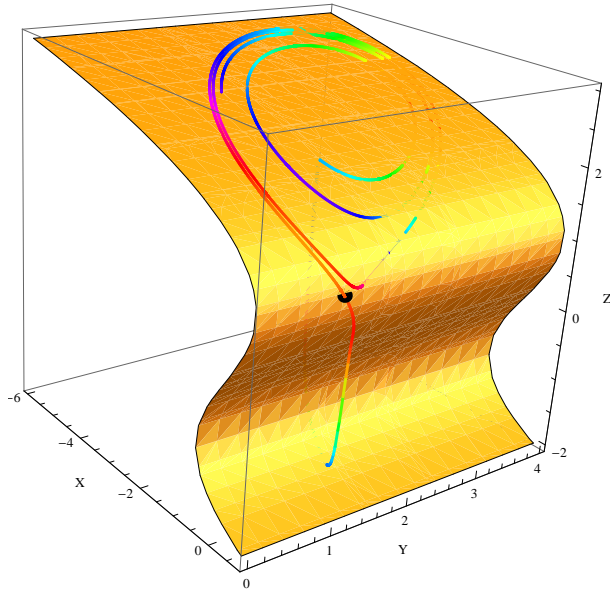}}
\caption{Canards solutions of Chua's system (\ref{eq10}).} \label{Fig1}
\end{figure}

\newpage

In Fig. 2 ``canards solutions'' winding around the \textit{pseudo-singular saddle point} $M(2/3, 1, 1)$ have been plotted for various values of parameter $\alpha$ in the $(z,x)$ phase plane for $\varepsilon = 1/20$.

\begin{figure}[htbp]
  \begin{center}
    \begin{tabular}{ccc}
      \includegraphics[width=6cm,height=6cm]{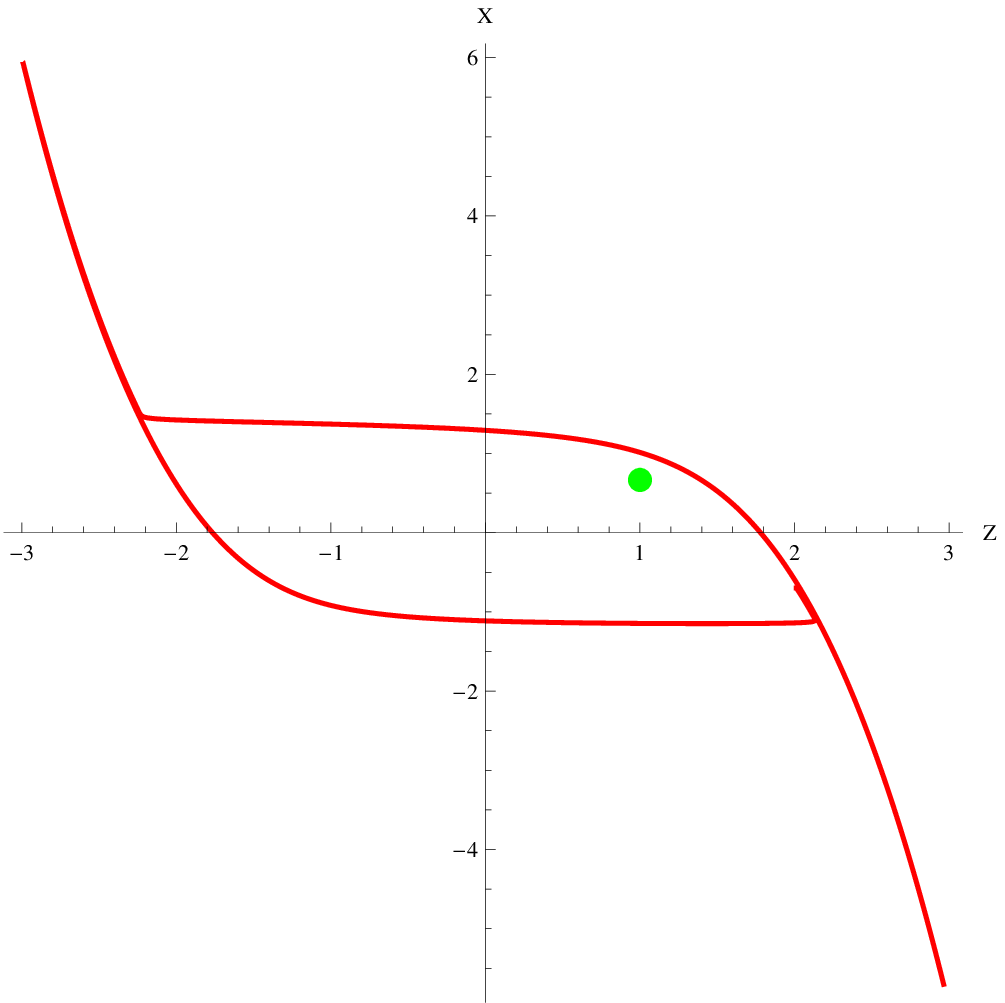} & ~~~~~~ &
      \includegraphics[width=6cm,height=6cm]{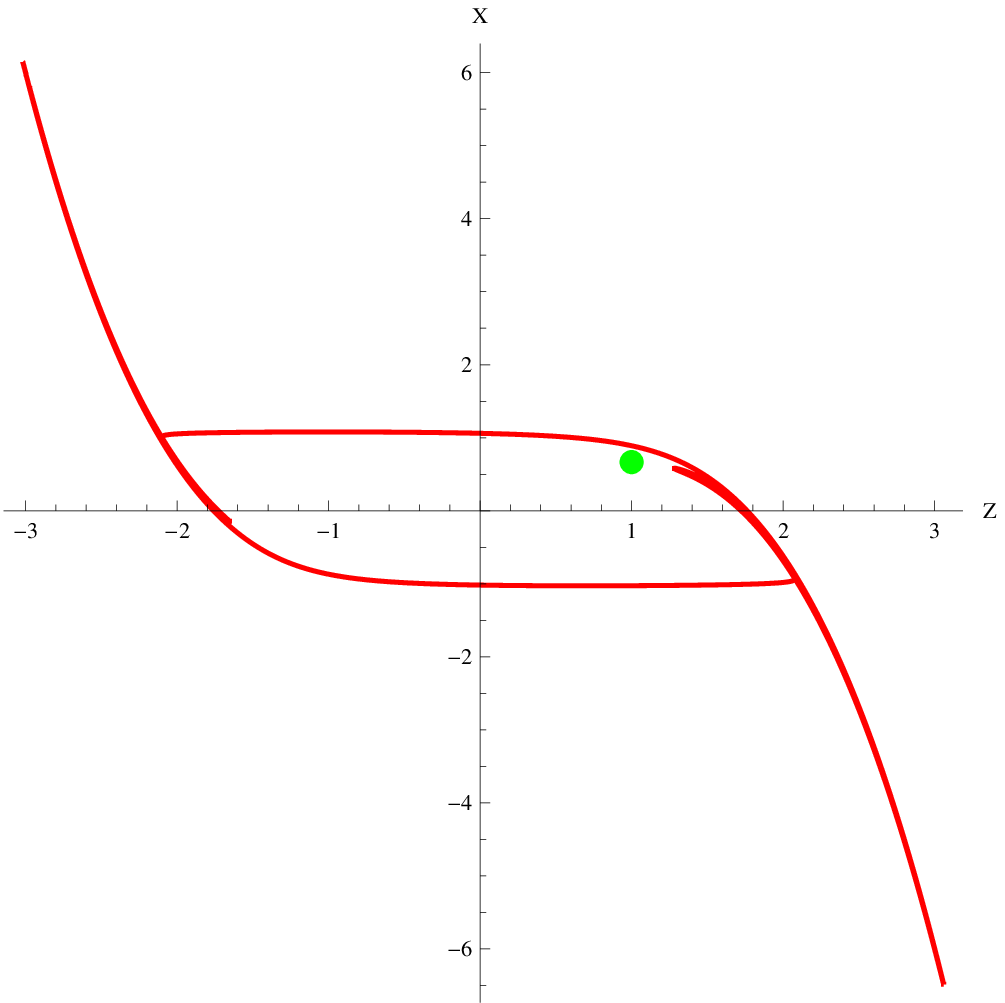} \\
      (a) $\alpha = 0.45$ & & (b) $\alpha = 0.35$ \\[0.5cm]
      \includegraphics[width=6cm,height=6cm]{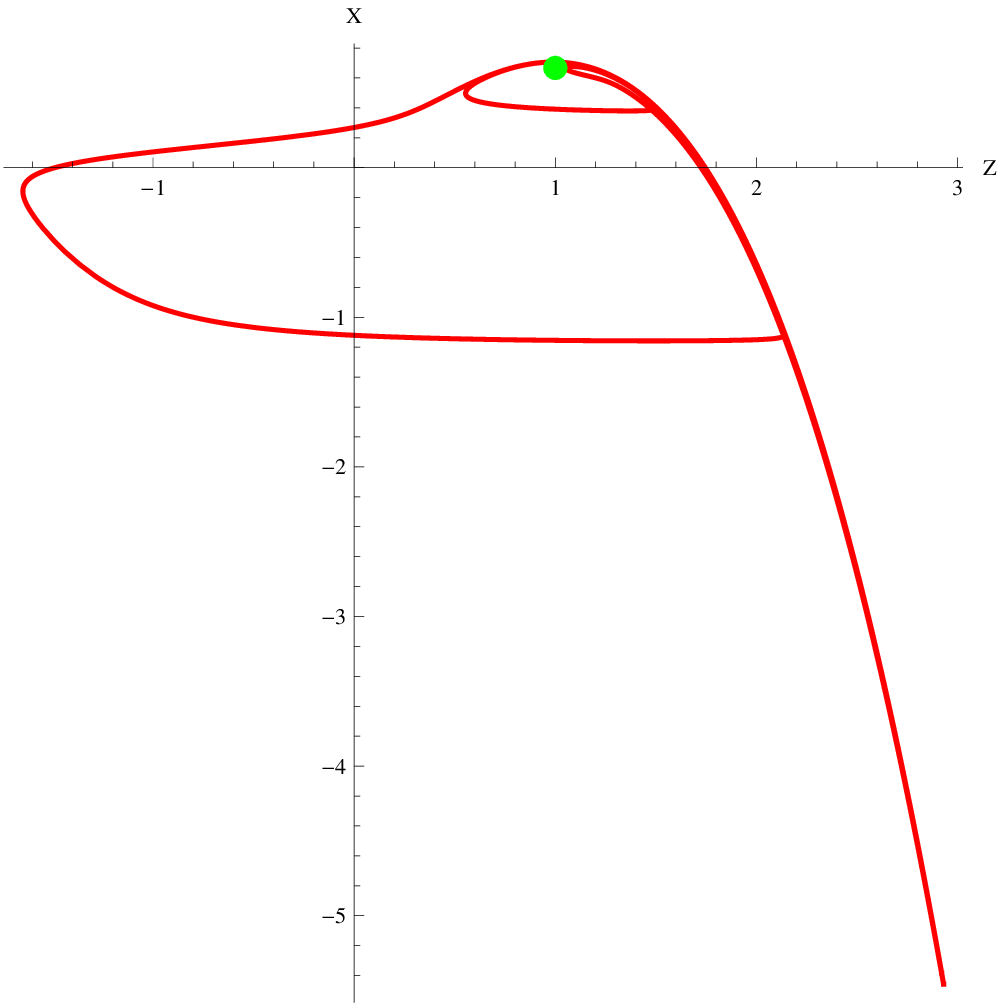} & ~~~ &
      \includegraphics[width=6cm,height=6cm]{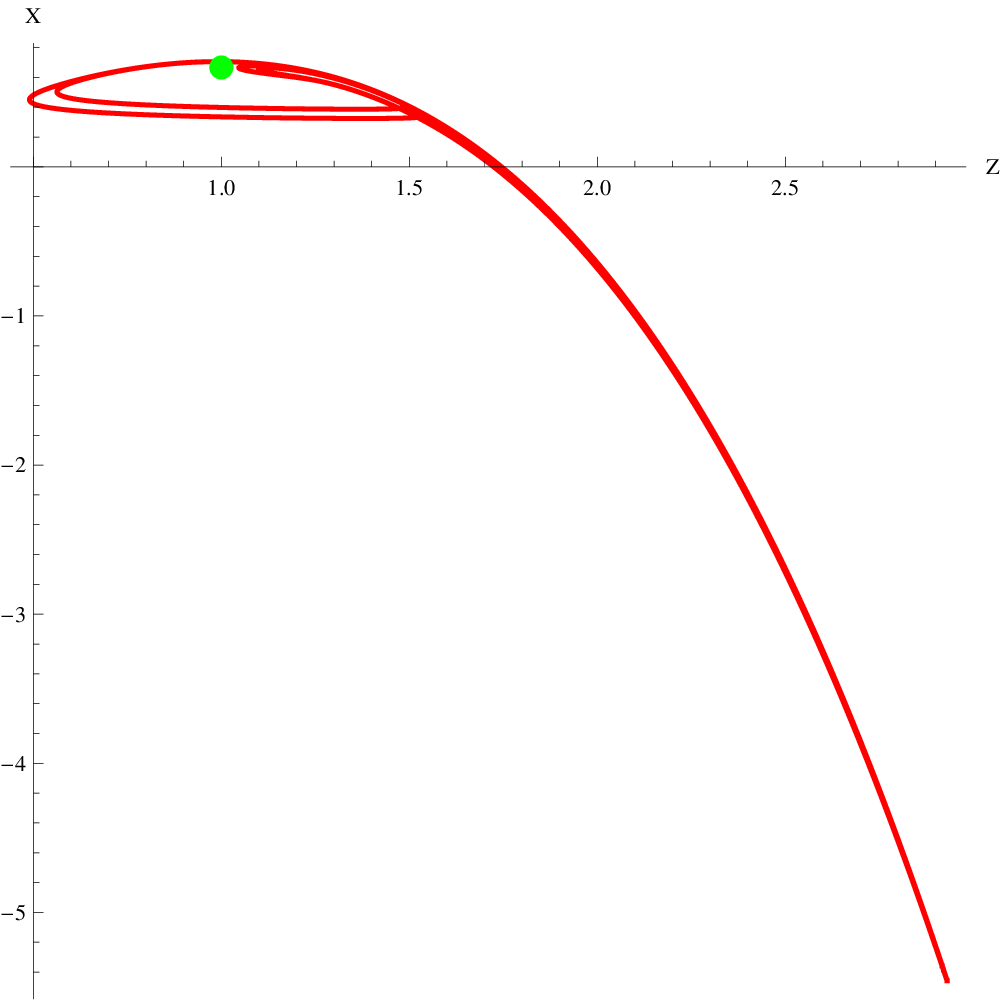} \\
      (c) $\alpha = 0.2571389636$ & & (d) $\alpha = 0.2571389$ \\[-0.5cm]
    \end{tabular}
    \vspace{0.1in}
    \caption{Canards solutions of Chua's system (\ref{eq10}).}
    \label{fig2}
  \end{center}
  \vspace{-0.25cm}
\end{figure}

\begin{remark}
Let's notice that we would have obtained the same kind of figures with the \textit{pseudo-singular saddle point} $M(-2/3, -1, -1)$ due to the symmetry of the system {\rm (\ref{eq10})}.
\end{remark}

\newpage

\section{Four-dimensional singularly perturbed systems}
\label{Sec4}

For four-dimensional \textit{singularly perturbed dynamical systems} we have three cases: $(p,m) = (3,1)$, $(p,m) = (2,2)$ and $(p,m) = (1,3)$. In this work we will only focus on the former case which will be subject to a special analysis allowing to extend Beno\^{i}t's Theorem \ref{theo4} to dimension four. So, in the case: $(p,m) = (3,1)$ four-dimensional \textit{singularly perturbed dynamical systems} may be defined as:

\begin{equation}
\label{eq15}
\begin{aligned}
 \dot{x_1} & = f_1 \left( x_1, x_2, x_3, y_1 \right), \hfill \\
 \dot{x_2} & = f_2 \left( x_1, x_2, x_3, y_1 \right), \hfill \\
 \dot{x_3} & = f_3 \left( x_1, x_2, x_3, y_1 \right), \hfill \\
 \varepsilon \dot{y_1} & = g_1 \left( x_1, x_2, x_3, y_1 \right), \hfill
\end{aligned}
\end{equation}

\smallskip

where $\vec{x}= (x_1, x_2, x_3)^t \in \mathbb{R}^3$, $\vec{y}= (y_1) \in \mathbb{R}^1$, $0 < \varepsilon << 1$, and the functions $f_i$ and $g_1$ are assumed to be $C^2$ functions of $(x_1, x_2, x_3, y_1)$.\\

The definitions of \textit{fold}, \textit{cusp} and \textit{pseudo-singular fixed points} may be extended to dimension four.

\subsection{Fold, cusp and pseudo-singular points}\hfill\\

Let's propose the following definitions

\begin{definition}\label{def3} \hfill \\

\smallskip

The location of the points where $\partial_{y_1} g_1\left(x_1,x_2, x_3, y_1 \right) = p\left(x_1,x_2, x_3, y_1 \right) = 0$ and $g_1\left(x_1,x_2,x_3,y_1 \right) = 0$ is called the \textit{fold}.
\end{definition}

The \textit{cofold} is still defined as the projection, if it exists, of the fold line onto $S_0$ along the
$y_1$-direction.

\smallskip

As previously system (\ref{eq15}) may have various types of singularities.

\begin{definition} \label{def4} \hfill \\

\begin{itemize}

\item[-] The fold is the set of points where the \textit{slow manifold}
is tangent to the $y_1$-direction.

\item[-] The cusp is the set of points where the \textit{fold} is
tangent to the $y_1$-direction.

\item[-] The stationary points are not on the \textit{fold} according to genericity assumptions.

\item[-] The pseudo-singular points are defined as the location of the points where

\begin{equation}
\label{eq16}
\begin{aligned}
& g_1\left(x_1,x_2,x_3,y_1 \right) = 0,\\
& \frac{\partial g_1\left(x_1,x_2,x_3,y_1 \right)}{\partial y_1} = 0, \\
& \frac{\partial g_1}{\partial x_1} f_1 + \frac{\partial g_1}{\partial x_2} f_2 + \frac{\partial g_1}{\partial x_3} f_3= 0.
\end{aligned}
\end{equation}

\end{itemize}

\end{definition}

\newpage

Again, following Beno\^{i}t [1983]:\\

\begin{itemize}

\item[-] the first condition indicates that the point belongs to the \textit{slow manifold},

\item[-] the second condition means that the point is on the \textit{fold},

\item[-] the third condition shows that the projection of the vector field on the
$(x_1,x_2)$-plane is tangent to the \textit{fold}.

\end{itemize}

\subsection{Reduced vector field}\hfill \\

If $x_1$ can be expressed as an implicit function of $x_2$, $x_3$ and $y_1$ defined by $g_1\left(x_1,x_2, x_3, y_1 \right) = 0$, the ``reduced normalized vector field'' reads:

\begin{equation}
\label{eq17}
\begin{aligned}
 \dot{x_2} & = - f_2 \left( x_1, x_2, x_3, y_1 \right) \frac{\partial g_1}{\partial y_1}\left( x_1, x_2, x_3, y_1 \right), \hfill \\
 \dot{x_3} & = - f_3 \left( x_1, x_2, x_3, y_1 \right) \frac{\partial g_1}{\partial y_1}\left( x_1, x_2, x_3, y_1 \right), \hfill \\
 \dot{y_1} & = \frac{\partial g_1}{\partial x_1} f_1 + \frac{\partial g_1}{\partial x_2} f_2 + \frac{\partial g_1}{\partial x_3} f_3.
\end{aligned}
\end{equation}

\subsection{Reduced vector field method}\hfill \\

By using the classification of fixed points of three-dimensional dynamical systems based on the sign of the eigenvalues, we can characterize the nature of the \textit{pseudo-singular point} $M$ of the ``reduced vector field'' (\ref{eq16}). Let's note $\Delta$ and $T$ respectively the determinant and the trace of the functional Jacobian matrix associated with system (\ref{eq16}) and $S = \sum \limits_{i=1}^3 J_{ii} $ where $J_{ii}$ is the minor obtained by removing the $i^{th}$ row and the $i^{th}$ column in the functional Jacobian matrix.
The discriminant of the \textit{characteristic polynomial} of the functional Jacobian matrix reads:

\[
R = 4P^3 + 27 Q^2 \mbox{\quad with \quad} P = S - \dfrac{T^2}{3} \mbox{\quad and \quad} Q = -\dfrac{2T^3}{27} + \dfrac{T S}{3} - \Delta
\]

Then, the \textit{pseudo-singular point} $M$:

\begin{itemize}

\item  a \textit{saddle} if and only if $R < 0$, i.e. $S < \dfrac{T^2}{3}$ and $\Delta < 0$.

\item  a \textit{node} if and only if $R < 0$ and $\Delta > 0$.

\item  a \textit{focus} if and only if $R > 0$.

\end{itemize}

Thus, we can extend Beno\^{i}t's Theorem \ref{theo4} to dimension four.

\begin{theorem}
\label{theo5} \hfill \\
If the ``reduced vector field'' {\rm (\ref{eq17})} has a pseudo-singular point of saddle type\footnote{In dimension three, a saddle point is a singular point having its three eigenvalues real but ``not all with the same sign''. See Poincar\'{e} [1886, p. 154].}, then system {\rm (\ref{eq15})} exhibits a canard solution which evolves from the attractive part of the slow manifold towards its repelling part.

\end{theorem}

\begin{proof}
Proof is based on the same arguments as previously.
\end{proof}

\subsection{Chua's system}\hfill \\

Let's consider the system introduced by Thamilmaran \textit{et al.} [2004]:

\begin{equation}
\label{eq18}
\begin{aligned}
 \dot{x} & = \beta_1 \left( z - x - u \right), \hfill \\
 \dot{y} & = \beta_2 z, \hfill \\
 \dot{z} & = -\alpha_2 z - y - x, \hfill \\
 \varepsilon \dot{u} & = x - k(u), \hfill
\end{aligned}
\end{equation}

where $k(u) = c_1 u^3 + c_2 u$, $\varepsilon = 1/\alpha_1$, $\alpha_2$, $c_{1,2}$ and $\beta_{1,2}$ are constant.\\

According to Eq. (\ref{eq17}) the reduced vector field reads:

\begin{equation}
\label{eq19}
\begin{aligned}
 \dot{y} & = - \beta_2 \left( - 3c_1 u^2 - c_2 \right)z, \hfill \\
 \dot{z} & = - \left( - 3c_1 u^2 - c_2 \right) \left(-y - c_1 u^3 - c_2 u - \alpha_2 z \right), \hfill \\
 \dot{u} & = \beta_1 \left( - u + z - c_1 u^3 - c_2 u \right).
\end{aligned}
\end{equation}

By Def. \ref{def4} the singularly perturbed dynamical system (\ref{eq18}) admits:

\[
M(0, \pm \frac{1}{3} \sqrt{\frac{- c_2}{3c_1}} (3 + 2 c_2), \pm \sqrt{\frac{-c_2}{3c_1}})
\]

as \textit{pseudo-singular points}. From the functional Jacobian matrix of system (\ref{eq19}) evaluated at $M$ we compute the \textit{characteristic polynomial} from which we deduce that:

\[
R=-\dfrac{4}{27} \beta_1^3 c_2^2 (3 \alpha_2 + 2 c_2 (1 + \alpha_2))^2 (8 c_2 (3 \alpha_2 + 2 c_2 (1 + \alpha_2)) + 3 \beta_1).
\]

\smallskip

In the parameter set used in system (\ref{eq18}) $\beta_1 > 0$ and $c_2 < 0$.\\

So, $R < 0$ provided that:

\[
8 c_2 (3 \alpha_2 + 2 c_2 (1 + \alpha_2)) + 3 \beta_1 < 0 \quad \Leftrightarrow \quad \alpha_2 < \dfrac{-16 c_2^2 - 3 \beta_1}{8 c_2 (3 + 2 c_2)}
\]

From the functional Jacobian matrix we also deduce that:

\[
\Delta = 0
\]

This implies that one of the three (real) eigenvalues (say $\lambda_1$) is null. So, in order to have a \textit{pseudo-singular saddle point} the two remaining eigenvalues (say $\lambda_2$ and $\lambda_3$) must be of different sign.\\

But, since $S = \lambda_1\lambda_2 + \lambda_1\lambda_3 + \lambda_2\lambda_3$ it means that $S = \lambda_2\lambda_3 < 0$. Thus, we may have

\[
S = \lambda_{2} \lambda_{3} = -\frac{2}{3}  \beta_1 c_2 (3 \alpha_2 + 2 c_2 (1 + \alpha_2)) < 0 \quad \Leftrightarrow \quad \alpha_2 < \dfrac{-2 c_2}{3 + 2 c_2}
\]

\smallskip

Combining the two required conditions, i.e., $R < 0$ and $\Delta < 0$ ($S < 0$ in this case) we find that:

\[
\alpha_2 < \dfrac{-2 c_2}{3 + 2 c_2} < \dfrac{-16 c_2^2 - 3 \beta_1}{8 c_2 (3 + 2 c_2)} = \dfrac{-2 c_2}{3 + 2 c_2} + \dfrac{- 3 \beta_1}{8 c_2 (3 + 2 c_2)}.
\]

\smallskip

So, according to Th. \ref{theo5}, if $\alpha_2 < -2c_2/(3 + 2 c_2)$, then $M$ is a \textit{pseudo-singular saddle point} and
so system (\ref{eq18}) exhibits canards solution.\\

Thus, ``canards'' solutions are observed in Chua's system (\ref{eq18}) for $\alpha_2 < -2c_2 / (3 +2 c_2)$ as exemplified in Fig. 3 in which such solutions passing through the \textit{pseudo-singular saddle point} $M$ have been plotted for parameter set $\varepsilon = 1 / \alpha_1 = 1 / 10.1428 = 0.098592$ ; $\alpha_2 = 0.9$ ; $\beta_1 = 0.121$ ; $\beta_2 = 0.0047$ ; $c_1 = 0.393781$ ; $c_2 = -0.72357$ in the $(u,z,x)$ phase space.

\begin{figure}[htbp]
\centerline{\includegraphics[width=10cm,height=10cm]{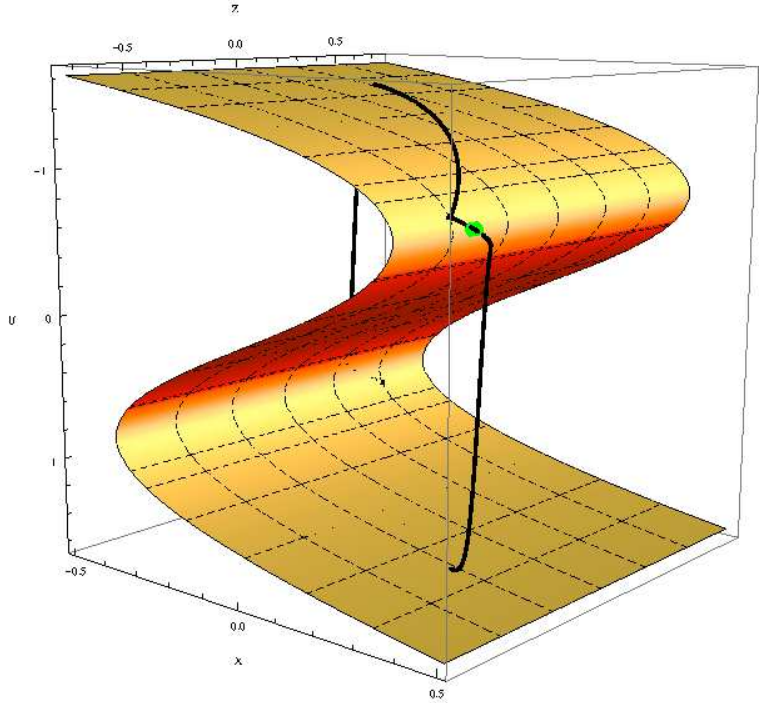}}
\caption{Canards solutions of Chua's system (\ref{eq18}).} \label{Fig3}
\end{figure}

\newpage

\section{Flow curvature method}
\label{Sec5}

A new approach called \textit{Flow Curvature Method} based on the use of \textit{Differential Geometry} properties of
\textit{curvatures} has been recently developed by Ginoux \textit{et al.} [2008] and Ginoux [2009]. According to this method, the highest \textit{curvature of the flow}, i.e. the $(n - 1)^{th}$ \textit{curvature} of \textit{trajectory curve} integral of $n$-dimensional dynamical system defines a \textit{manifold} associated with this system and called \textit{flow curvature manifold}.

\begin{definition}
\label{def5} \hfill \\
The location of the points where the $(n - 1)^{th}$ \textit{curvature of the flow}, i.e. the \textit{curvature of the trajectory curve} $\vec {X}$, integral of any $n$-dimensional singularly perturbed dynamical system {\rm (\ref{eq5})} vanishes, defines a $(n - 1)$-dimensional \textit{flow curvature manifold} the equation of which is:

\begin{equation}
\label{eq20} \phi ( {\vec {X}} ) = \dot { \vec
{X} } \cdot ( {\ddot {\vec {X}} \wedge \dddot{\vec {X}}\wedge \ldots
\wedge \mathop {\vec {X}}\limits^{\left( n \right)} } ) = \det( {
\dot {\vec {X}},\ddot {\vec {X}}, \dddot{\vec {X}},\ldots ,\mathop
{\vec {X}}\limits^{\left( n \right)} } ) = 0
\end{equation}

where $\mathop {\vec {X}}\limits^{\left( n \right)}$ represents the time derivatives up to order $n$ of $\vec {X}  = (\vec{x}, \vec{y})^t$.

\end{definition}

\subsection{Three-dimensional singularly perturbed systems} \hfill \\

According to the \textit{Flow Curvature Method} the \textit{flow curvature manifold} of the \textit{reduced vector field} (\ref{eq9}) is defined by:

\begin{equation}
\label{eq21} \phi ( \vec{X} ) = \det( \dot {\vec {X}},\ddot {\vec {X}} ) = 0
\end{equation}

where $\vec {X}  = (x_2, y_1)^t$.

\smallskip

We suppose that the \textit{flow curvature manifold} $\phi(x_2, y_1)$ admits at $M(x_2^*,y_1^*)$ an \textit{extremum} such that: $\partial_{x_2} \phi = \partial_{y_1} \phi =0$.\\

The Hessian matrix of the manifold $\phi (x_2, y_1)$ is defined, provided that all the second partial derivatives of $\phi$ exist, by

\begin{equation}
\label{eq22}
H_{\phi (x_2, y_1)} = \begin{pmatrix}
\dfrac{\partial^2\phi}{\partial x^2_2 } & \dfrac{\partial^2\phi}{\partial x_2 \partial y_1 } \vspace{6pt} \\
\dfrac{\partial^2\phi}{\partial y_1 \partial x_2 }  &  \dfrac{\partial^2\phi}{\partial y^2_1 }\\
\end{pmatrix}.
\end{equation}

\smallskip

Then, according to the so-called \textit{Second Derivative Test} (see for example Thomas \& Finney [1992]) and by noticing

\begin{equation}
\label{eq23}
D_1 = \dfrac{\partial^2 \phi}{\partial x^2_2} \mbox{,} \quad
D_2 =
\begin{vmatrix}
\dfrac{\partial^2 \phi}{\partial x^2_2} & \dfrac{\partial^2 \phi}{\partial x_2 \partial y_1} \vspace{4pt} \\
\dfrac{\partial^2 \phi}{\partial y_1 \partial x_2} & \dfrac{\partial^2 \phi}{\partial y^2_1}
\end{vmatrix}
\end{equation}

\newpage

if $D_2 \neq 0$, the \textit{flow curvature manifold} (\ref{eq21}) admits $M(x_2^*,y_1^*)$ as a \\

\begin{itemize}

\item \textit{local minimum}, if and only if $(D_1,D_2) = (+,+)$,

\item \textit{local maximum}, if and only if $(D_1,D_2) = (-,+)$, and

\item \textit{saddle-point}, if and only if $D_2 < 0$.\\

\end{itemize}

Thus, we have the following proposition:

\begin{proposition}
\label{prop1} \hfill \\
If the flow curvature manifold of the ``reduced vector field'' {\rm (\ref{eq9})} admits a pseudo-singular point of saddle-type, then system {\rm (\ref{eq7}) } exhibits a canard solution which evolves from the attractive part of the slow manifold towards its repelling part.
\end{proposition}

\begin{proof}
According to the theorem of Hartman-Grobman [1964] the flow of any dynamical system (\ref{eq9}) is \textit{locally topologically conjugated} to the flow of the linearized system in the vicinity of \textit{fixed points}. So, let's consider the linearized system in the basis of the eigenvectors:

\[
\begin{aligned}
 \dot{x_1} & = \lambda_1 x_1, \hfill \\
 \dot{x_2} & = \lambda_2 x_2.
\end{aligned}
\]

where $\lambda_{1,2}$ are the eigenvalues of the functional Jacobian matrix.
The \textit{flow curvature manifold} (\ref{eq21}) associated with this linearized system reads:

\[
\phi ( \vec{X} ) = \det( \dot {\vec {X}},\ddot {\vec {X}} ) = x_1 x_2 \lambda_1 \lambda_2 (\lambda_2 -  \lambda_1)
\]

\smallskip

Then, it's easy to check that the determinant $D_2$ of the Hessian may be written as:

\[
D_2 = - \Delta^2 \left( T^2 - 4 \Delta \right)
\]

\smallskip

from which we deduce that if $D_2 < 0$ then $M$ is a \textit{saddle-point} provided that $T = \lambda_1 + \lambda_2 $ and $\Delta = \lambda_1 \lambda_2$ are not null. \end{proof}

\begin{remark}
This idea corresponds to topographic system introduced by Poincar\'{e} {\rm [1881-1886]} in his memoirs entitled: ``Sur les courbes d\'{e}finies par une \'{e}quation diff\'{e}rentielle''. Topographic system consists in using level set such as $f\left(x_1, x_2 \right) = constant$ surrounding fixed points in order to define their nature (node, saddle, foci) and their stability. Moreover, Prop. \ref{prop1} leads to the same kind of result as that obtained by Szmolyan \textit{et al.} {\rm [2001]} but without needing to make a change of variables on system {\rm (\ref{eq7})} other than that proposed by Beno\^{i}t {\rm [1983, 2001]}.
\end{remark}

The \textit{Flow Curvature Method} has been successfully used by the Ginoux \textit{et al.} [2011] for computing the bifurcation parameter value leading to a canard explosion in dimension two already obtained according to the so-called \textit{Geometric Singular Perturbation Method}.

\newpage

\subsection{Chua's system}\hfill \\

Let's consider again the system (\ref{eq10}) of Itoh \& Chua [1992]

\[
\begin{aligned}
\dot{x} & = z - y, \\
\dot{y} & = \alpha(x + y), \\
\varepsilon \dot{z} & = -x - k(z),
\end{aligned}
\]

where $k(z) = z^3/3 - z$ and $\alpha$ is a constant.

The reduced vector field (\ref{eq11}):

\[
\begin{aligned}
\dot{y} & = \alpha k'(z)(-k(z) + y) = \alpha (z^2 - 1)\left(- \dfrac{z^3}{3} + z + y\right), \\
\dot{z} & = y - z.
\end{aligned}
\]

The \textit{flow curvature manifold} (\ref{eq21}) associated with this reduced vector field reads:

\[
\begin{aligned}
\phi(y,z) = & \dfrac{\alpha}{9} [-3 (y - z) (6 y^2 z + 4 z^3 (-2 + z^2) + y (-6 + 9 z^2 - 5 z^4)) \\
& + z (-6 + z^2) (-1 + z^2)^2 (-3 y - 3 z + z^3) \alpha ] = 0.
\end{aligned}
\]

Proposition \ref{prop1} enables to state that the determinant of the Hessian evaluated at $(y^*, z^*)=(\pm 1, \pm 1)$ becomes

\[
D_2 = -\dfrac{100}{27} \alpha^2 (3 + 40 \alpha),
\]

from which one deduces that if $3 + 40 \alpha > 0$ then $M$ is a \textit{pseudo-singular saddle point} and so systems (\ref{eq10})
exhibits a canard solution. Thus, we find Beno\^{i}t's result according to the \textit{Flow Curvature Method}.

\subsection{Four-dimensional singularly perturbed systems}\hfill \\

According to the \textit{Flow Curvature Method} the \textit{flow curvature manifold} of the \textit{reduced vector field} (\ref{eq17}) is defined by:

\begin{equation}
\label{eq24} \phi ( \vec{X} ) = \det( \dot {\vec {X}},\ddot {\vec {X}}, \dddot {\vec {X}} ) = 0
\end{equation}

where $\vec {X}  = (x_2, x_3, y_1)^t$.

\smallskip

We suppose that the \textit{flow curvature manifold} $\phi(x_2,x_3, y_1)$ admits at $M(x_2^*,x_3^*,y_1^*)$ an \textit{extremum} such that: $\partial_{x_2} \phi = \partial_{x_3} \phi = \partial_{y_1} \phi =0$.\\

The Hessian matrix of the manifold $\phi (x_2,x_3,y_1)$ is defined, provided that all the second partial derivatives of $\phi$ exist, by

\begin{equation}
\label{eq25}
H_{\phi (x_2,x_3,y_1)} = \begin{vmatrix}
\dfrac{\partial^2\phi}{\partial x^2_2 } & \dfrac{\partial^2\phi}{\partial x_2 \partial x_3 }  & \dfrac{\partial^2\phi}{\partial x_2 \partial y_1 } \vspace{6pt} \\
\dfrac{\partial ^2\phi}{\partial x_3 \partial x_2 }  &  \dfrac{\partial^2\phi}{\partial x^2_3 } & \dfrac{\partial^2\phi}{\partial x_3 \partial y_1 } \vspace{6pt} \\
\dfrac{\partial ^2\phi}{\partial y_1 \partial x_2 } & \dfrac{\partial^2\phi}{\partial y_1 \partial x_3 } & \dfrac{\partial^2\phi}{\partial y^2_1 }
\end{vmatrix}.
\end{equation}

Then, according to the so-called \textit{Second Derivative Test} and while noticing $D_1$ the determinant of the upper left $1 \times 1$ submatrix of $H_{\phi}$, $D_2$ the determinant of the $2 \times 2$ matrix of $H_{\phi}$ defined as:

\begin{equation}
\label{eq26}
D_1 = \dfrac{\partial^2 \phi}{\partial x^2_2} \mbox{,} \quad
D_2 =
\begin{vmatrix}
\dfrac{\partial^2 \phi}{\partial x^2_2} & \dfrac{\partial^2 \phi}{\partial x_2 \partial x_3} \vspace{4pt} \\
\dfrac{\partial^2 \phi}{\partial x_3 \partial x_2} & \dfrac{\partial^2 \phi}{\partial x^2_3}
\end{vmatrix},
\end{equation}

by $D_3$ the determinant of the $3 \times 3$ matrix of $H_{\phi}$ defined as:

\begin{equation}
\label{eq27}
D_3 = \begin{vmatrix}
\dfrac{\partial^2\phi}{\partial x^2_2 } & \dfrac{\partial^2\phi}{\partial x_2 \partial x_3 }  & \dfrac{\partial^2\phi}{\partial x_2 \partial y_1 } \vspace{6pt} \\
\dfrac{\partial ^2\phi}{\partial x_3 \partial x_2 }  &  \dfrac{\partial^2\phi}{\partial x^2_3 } & \dfrac{\partial^2\phi}{\partial x_3 \partial y_1 } \vspace{6pt} \\
\dfrac{\partial ^2\phi}{\partial y_1 \partial x_2 } & \dfrac{\partial^2\phi}{\partial y_1 \partial x_3 } & \dfrac{\partial^2\phi}{\partial y^2_1 }
\end{vmatrix},
\end{equation}

\smallskip

if $D_3 \neq 0$, the \textit{flow curvature manifold} (\ref{eq24}) admits $M(x_2^*,x_3^*,y_1^*)$ as a \\

\begin{itemize}

\item \textit{local minimum}, if and only if $(D_1,D_2,D_3) = (+,+,+)$

\item \textit{local maximum}, if and only if $(D_1,D_2,D_3) = (-,+,-)$

\item \textit{saddle-point}, in all other cases.\\

\end{itemize}

So, we have the following proposition.

\begin{proposition}
\label{prop2} \hfill \\
If the flow curvature manifold of the ``reduced vector field'' {\rm (\ref{eq17})} admits a pseudo-singular saddle-point, then system {\rm (\ref{eq15})} exhibits a canard solution which evolves from the attractive part of the slow manifold towards its repelling part.
\end{proposition}

\begin{proof}
According to Hartman-Grobman's Theorem [1964] the flow of any dynamical system (\ref{eq17}) is \textit{locally topologically conjugated} to the flow of the linearized system in the vicinity of \textit{fixed points}. So, let's consider the linearized system in the basis of the eigenvectors:

\[
\begin{aligned}
 \dot{x_1} & = \lambda_1 x_1, \hfill \\
 \dot{x_2} & = \lambda_2 x_2, \hfill \\
 \dot{x_3} & = \lambda_3 x_3.
\end{aligned}
\]

where $\lambda_{1,2,3}$ are the eigenvalues of the functional Jacobian matrix.
The \textit{flow curvature manifold} (\ref{eq24}) associated with this linearized system reads:

\[
\phi ( \vec{X} ) = \det( \dot {\vec {X}},\ddot {\vec {X}},\dddot {\vec {X}} ) = x_1 x_2 x_3 \lambda_1 \lambda_2 \lambda_3 (\lambda_2 -  \lambda_1)(\lambda_1 -  \lambda_3)(\lambda_2 -  \lambda_3).
\]

\newpage

Then, it's easy to check that the determinant $D_3$ of the Hessian evaluated at $M$ is such that\footnote{The symbol $\propto$ means proportional to.}:

\[
D_3 \propto - 2\Delta^2 R,
\]

\smallskip

from which we deduce that if $D_3$ is positive, i.e. $R < 0$, then $M$ is a \textit{saddle-point} provided that $(D_1,D_2) \neq (+,+)$. \end{proof}

\subsection{Chua's system}\hfill \\

Let's consider again the system  (\ref{eq18}) of Thamilmaran \textit{et al.} [2004]:

\[
\begin{aligned}
 \dot{x} & = \beta_1 \left( z - x - u \right), \hfill \\
 \dot{y} & = \beta_2 z, \hfill \\
 \dot{z} & = -\alpha_2 z - y - x, \hfill \\
 \varepsilon \dot{u} & = x - k(u). \hfill
\end{aligned}
\]

\smallskip

where $k(u) = c_1 u^3 + c_2 u$, $\varepsilon = 1/\alpha_1$, $\alpha_2$, $c_{1,2}$ and $\beta_{1,2}$ are constant.\\

The reduced vector field (\ref{eq17})  reads:

\[
\begin{aligned}
 \dot{y} & = - \beta_2 \left( - 3c_1 u^2 - c_2 \right)z, \hfill \\
 \dot{z} & = - \left( - 3c_1 u^2 - c_2 \right) \left(-y - c_1 u^3 - c_2 u - \alpha_2 z \right), \hfill \\
 \dot{u} & = \beta_1 \left( - u + z - c_1 u^3 - c_2 u \right).
\end{aligned}
\]

\smallskip

The \textit{flow curvature manifold} (\ref{eq24}) associated with this reduced vector field reads\footnote{This equation which is too large to be presented here is available at http://ginoux.univ-tln.fr.}:

\[
\phi ( \vec{X} ) = \det( \dot {\vec {X}},\ddot {\vec {X}}, \dddot {\vec {X}} ) = \phi ( y,z,u ) = 0.
\]

\smallskip

By considering that the parameter set of this system is such that $\beta_2 \ll 1$ and according to proposition \ref{prop2} we find that:

\[
D_1 \propto c_2 (3\alpha_2 +   2 c_2 (1 + \alpha_2))^2,
\]

\[
D_2  \propto - (6 c_2 \alpha_2 + 4c_2^2(1 + \alpha_2) + \beta_1),
\]

\[
D_3 \propto (6 c_2 \alpha_2 + 4c_2^2(1 + \alpha_2))(6 c_2 \alpha_2 + 4c_2^2(1 + \alpha_2) + \beta_1) P(\alpha_2, c_2),
\]

\smallskip

where $P(\alpha_2, c_2)$ is a positive quadratic polynomial in $\alpha_2$.\\

Since $c_2 < 0$, we deduce that $M$ is a \textit{saddle point} provided that

\[
\alpha_2 < \dfrac{-2 c_2}{3 + 2 c_2}.
\]

\smallskip

Thus, we find Beno\^{i}t's result according to the \textit{Flow Curvature Method}.

\section{Discussion}

In this work Beno\^{i}t's theorem for the generic existence of ``canards'' solutions in \textit{singularly perturbed dynamical systems} of dimension three with one fast variable has been extended to those of dimension four. Then, it has been established that this result can be found according to the \textit{Flow Curvature Method}. The Hessian of the \textit{flow curvature manifold} and the so-called \textit{Second Derivative Test} enabled to characterize the nature of the \textit{pseudo-singular saddle points}. Applications to Chua's cubic model of dimension three and four highlighted the existence of ``canards'' solutions in such systems. According to Prof. Eric Beno\^{i}t (personal communications) the cases $(p,m)=(3,1)$ and $(p,m) = (2,2)$ for which his theorems [Beno\^{i}t, 1983, 2001] for canard existence at pseudo-singular points of saddle-type still holds have been completely analyzed while the case $p = 1$ and $m = 3$ remains an open problem since the fold becomes a two-dimensional manifold and the pseudo-singular fixed points become pseudo-singular curves. In this case, fold and cusps are defined according to the theory of surfaces singularities and are strongly related to Thom's catastrophe theory [Thom, 1989].

\section*{Acknowledgments}

First author would like to thank Prof. Martin Wechselberger for his fruitful
advices. Moreover, let's notice that our main result has been already established by Wechselberger [2012] who has extended \textit{canard theory} of singularly perturbed systems to the more general case of $k+m$-dimensional \textit{singularly perturbed systems} with $k$ \textit{slow} and $m$ \textit{fast} dimensions, with $k \geqslant 2$ and $m \geqslant 1$. The second author is supported by the grants MICIIN/FEDER
MTM 2008--03437, AGAUR 2009SGR410, and ICREA Academia and FP7-PEOPLE-2012-IRSES-316338.

\section*{References}

\hspace{0.2in} Andronov, A. A. {\&} Khaikin, S. E. [1937] \textit{Theory of oscillators}, I, Moscow (Engl. transl.,
Princeton Univ. Press, Princeton, N. J., 1949).\\

Arg\'{e}mi, J. [1978] ``Approche qualitative d'un probl\`{e}me de perturbations singuli\`{e}res dans
$\mathbb{R}^4$,'' in \textit{Equadiff 1978}, ed. R. Conti, G. Sestini, G. Villari, 330--340.\\

Beno\^{i}t, E., Callot, J.L., Diener,  F. \&  Diener, M. [1981] ``Chasse au canard,'' \textit{Collectanea Mathematica} {\bf 31--32} (1-3),
37--119.\\

Beno\^{i}t, E. [1982] ``Les canards de $\mathbb{R}^3$,'' {\it C.R.A.S.} t. 294, S\'{e}rie I, 483--488.\\

Beno\^{i}t, E. [1983] ``Syst\`{e}mes lents-rapides dans $\mathbb{R}^3$ et leurs canards,'' \textit{Soci\'{e}t\'{e} Math\'{e}matique de France}, {\it Ast\'{e}risque} (190--110), 159--191.\\

Beno\^{i}t, E. [2001] ``Perturbation singuli\`{e}re en dimension trois~: Canards en un point pseudosingulier
noeud,'' \textit{Bulletin de la Soci\'{e}t\'{e} Math\'{e}matique de France}, (129-1), 91--113.\\

Fenichel, N. [1971] ``Persistence and smoothness of invariant manifolds for flows,'' \textit{Ind. Univ. Math. J.} 21, 193--225.\\

Fenichel, N. [1974] ``Asymptotic stability with rate conditions,'' \textit{Ind. Univ. Math. J.} 23, 1109--1137.\\

Fenichel, N. [1977] ``Asymptotic stability with rate conditions II,'' \textit{Ind. Univ. Math. J.} 26, 81--93.\\

Fenichel, N. [1979] ``Geometric singular perturbation theory for ordinary differential equations,'' \textit{J. Diff. Eq.}, 53--98.\\

Ginoux, J.M., Rossetto,  B. \& Chua, L.O. [2008] ``Slow Invariant Manifolds as Curvature of the Flow of Dynamical
Systems,'' {\it Int. J. Bif. {\&} Chaos} 11 (18), 3409--3430.\\

Ginoux, J.M. [2009] \textit{Differential Geometry applied to Dynamical Systems}, World Scientific Series on Nonlinear Science, Series A {\bf 66} (World Scientific, Singapore).

Ginoux, J.M. \& Llibre, J. [2011] ``Flow Curvature Method applied to Canard Explosion,'' {\it J. Physics A: Math. Theor.}, 465203, 13pp.\\

Hartman, P. [1964] \textit{Ordinary Differential Equations}, J. Wiley and. Sons, New York.\\

Itoh, M. \& Chua, L.O. [1992] ``Canards and Chaos in Nonlinear Systems,'' \textit{Proc. of 1992 IEEE
International Symposium on Circuits and Systems}, San Diego, 2789--2792.\\

Jones, C.K.R.T. [1994] ``Geometric Singular Perturbation Theory in Dynamical Systems,''
\textit{Montecatini Terme}, L. Arnold, Lecture Notes in Mathematics, vol. 1609, Springer-Verlag, 44--118.\\

Kaper, T. [1999] ``An Introduction to Geometric Methods and Dynamical Systems Theory for Singular Perturbation Problems,'' in \textit{Analyzing multiscale phenomena using singular perturbation methods}, (Baltimore, MD, 1998), pages 85--131. Amer. Math. Soc., Providence, RI.\\

Lyapounov, A.M. [1892] ``The general problem of the stability of motion,'' Ph-D Thesis, St Petersbourg, (1892), reprinted in ``Probl\`{e}me g\'{e}n\'{e}ral de la stabilit\'{e} du mouvement,'' \emph{Ann. Fac. Sci.}, Toulouse 9, 203-474, (1907), reproduced in \textit{Ann. Math. Stud.}, \textbf{12}, (1949).\\

O'Malley, R.E. [1974] \textit{Introduction to Singular Perturbations}, Academic Press, New York.\\

O'Malley, R.E. [1991] \textit{Singular Perturbations Methods for Ordinary Differential Equations}, Springer-Verlag, New York.\\

Poincar\'{e}, H. [1881] ``Sur les courbes d\'{e}finies par une \'{e}quation diff\'{e}rentielle,'' {\it Journal de Math\'{e}matiques Pures et Appliqu\'{e}es}, 3\textsuperscript{o} s\'{e}rie, {\bf 7}, 375--422.\\

Poincar\'{e}, H. [1882] ``Sur les courbes d\'{e}finies par une \'{e}quation diff\'{e}rentielle,'' {\it Journal de Math\'{e}matiques Pures et Appliqu\'{e}es}, 3\textsuperscript{o} s\'{e}rie, {\bf 8}, 251--296.\\

Poincar\'{e}, H. [1885] ``Sur les courbes d\'{e}finies par une \'{e}quation diff\'{e}rentielle,''
{\it Journal de Math\'{e}matiques Pures et Appliqu\'{e}es}, 4\textsuperscript{o} s\'{e}rie, {\bf 1}, 167--244.\\

Poincar\'{e}, H. [1886] ``Sur les courbes d\'{e}finies par une \'{e}quation diff\'{e}rentielle,''
{\it Journal de Math\'{e}matiques Pures et Appliqu\'{e}es}, 4\textsuperscript{o} s\'{e}rie, {\bf 2}, 151--217.\\

Rossetto, B. [1986] ``Trajectoires lentes de syst\`{e}mes dynamiques lents-rapides,'' \textit{International Conference on Analysis and Optimization}, unpublished notes.\\

Szmolyan, P. \& Wechselberger, M. [2001] ``Canards in $\mathbb{R}^3$,'' \textit{J. Dif. Eqs.} 177, 419--453.\\

Takens, F. [1976] ``Constrained equations, a study of implicit differential equations and their discontinuous solutions,'' in \textit{Structural
stability, the theory of catastrophes and applications in the sciences}, \textit{Springer Lecture Notes in Math.}, \textbf{525}, 143--234.\\

Thamilmaran, K., Lakshmanan,  M. {\&} Venkatesan, A. [2004] ``Hyperchaos in a modified canonical Chua's circuit,'' {\it Int. J. Bifurcation and Chaos}, vol. 14, 221--243.\\

Thom, R. [1989] \textit{Structural Stability and Morphogenesis: An Outline of a General Theory of Models} Reading, MA: Addison-Wesley.\\

Thomas, Jr. G.B.  \& Finney, R.L. [1992] \textit{Maxima, Minima, and Saddle Points}, \S{}12.8 in Calculus and Analytic Geometry, 8$^{th}$ ed. Reading, MA: Addison-Wesley, 881-891.\\

Tikhonov, A.N. [1948] ``On the dependence of solutions of differential equations on a small parameter,'' \textit{Mat. Sbornik N.S.}, 31, 575--586.\\

Van der Pol, B. [1926] ``On relaxation-oscillations,'' {\it The London, Edinburgh, and Dublin Philosophical Magazine and Journal of Science}, {\bf 7} (2), 978--992.\\

Wechselberger, M.  [2005] ``Existence and Bifurcation of Canards in $\mathbb{R}^3$ in the case of a Folded Node,'' \textit{SIAM J. Applied Dynamical Systems} 4, 101--139.\\

Wechselberger, M. [2012] ``{A} propos de canards,'' \textit{Trans. Amer. Math. Soc.}, \textbf{364} (2012) 3289--330.\\

Zvonkin, A. K. \& Shubin, M. A. [1984] ``Non-standard analysis and singular perturbations of ordinary differential equations,'' {\it Uspekhi Mat. Nauk.} \textbf{39} 2 (236), 69--131.

%\end{multicols}
\end{document}